\documentclass[12pt]{amsart}
\usepackage{a4wide}
\usepackage[T1]{fontenc}
\usepackage{amssymb,amsmath,amsthm,latexsym}
\usepackage{mathrsfs}
\usepackage[usenames,dvipsnames]{color}
\usepackage{euscript}
\usepackage{graphicx}
\usepackage{mdwlist}
\usepackage{enumerate}
\usepackage{mathtools,dsfont,wasysym}
\usepackage{stmaryrd}
\usepackage[dvipsnames]{xcolor}

\DeclareFontFamily{U}{mathx}{\hyphenchar\font45}
\DeclareFontShape{U}{mathx}{m}{n}{
      <5> <6> <7> <8> <9> <10>
      <10.95> <12> <14.4> <17.28> <20.74> <24.88>
      mathx10
      }{}

\usepackage{todonotes}

\newtheorem{theorem}{Theorem}[section]
\newtheorem*{theoremA}{Theorem A}
\newtheorem*{theoremB}{Theorem B}

\newtheorem{lemma}[theorem]{Lemma}
\newtheorem{corollary}[theorem]{Corollary}
\newtheorem{proposition}[theorem]{Proposition}
\newtheorem{fact}[theorem]{Fact}
\theoremstyle{remark}
\newtheorem{remark}[theorem]{Remark}
\newtheorem{claim}[theorem]{Claim}
\theoremstyle{definition}
\newtheorem{definition}[theorem]{Definition}
\newtheorem{problem}[theorem]{Problem}

\numberwithin{equation}{section}
\makeatother

\newcommand{\nn}[1]{{\left\vert\kern-0.25ex\left\vert\kern-0.25ex\left\vert #1 
\right\vert\kern-0.25ex\right\vert\kern-0.25ex\right\vert}}
\newcommand{\flag}{\mathord{\upharpoonright}}
\renewcommand{\leq}{\leqslant}
\renewcommand{\geq}{\geqslant}
\DeclareMathOperator{\supp}{supp}
\DeclareMathOperator{\dens}{dens}
\DeclareMathOperator{\cf}{cf}

\newcounter{smallromans}

\newenvironment{romanenumerate}
{\begin{list}{{\normalfont\textrm{(\roman{smallromans})}}}%
  {\usecounter{smallromans}\setlength{\itemindent}{0cm}%
   \setlength{\leftmargin}{5.5ex}\setlength{\labelwidth}{5.5ex}%
   \setlength{\topsep}{.5ex}\setlength{\partopsep}{.5ex}%
   \setlength{\itemsep}{0.1ex}}}%
{\end{list}}

\newcommand{\R}{\mathbb{R}}

\newcommand{\e}{\varepsilon}

\newcounter{smallromansdash}

{\end{list}}

\newcounter{bigromans} 
{\end{list}}

\begin{document}
\title[Overcomplete sets in non-separable Banach spaces]{Overcomplete sets\\in non-separable Banach spaces}
\dedicatory{In memoriam: Paolo Terenzi (1940--2017)}

\author[T.~Russo]{Tommaso Russo}
\address[T.~Russo]{Department of Mathematics\\Faculty of Electrical Engineering\\Czech Technical University in Prague\\Technick\'a 2, 166 27 Praha 6\\ Czech Republic}
\email{russotom@fel.cvut.cz}

\author[J.~Somaglia]{Jacopo Somaglia}
\address[J.~Somaglia]{Dipartimento di Matematica ``F. Enriques'' \\Universit\`a degli Studi di Milano\\Via Cesare Saldini 50, 20133 Milano\\Italy}
\email{jacopo.somaglia@unimi.it}

\thanks{Part of the research presented in this paper was carried out during the second author's visit to the Institute of Mathematics of the Czech Academy of Sciences. We acknowledge with thanks GA\v{C}R project 17-27844S; RVO 67985840 for supporting the research stay. Research of the first-named author was supported  by the project International Mobility of Researchers in CTU CZ.02.2.69/0.0/0.0/16$\_$027/0008465. Research of the second-named author was supported  by Universit\`a degli Studi di Milano, Research Support Plan 2019. Finally, both authors were also supported by Gruppo Nazionale per l'Analisi Matematica, la Probabilit\`a e le loro Applicazioni (GNAMPA) of Istituto Nazionale di Alta Matematica (INdAM), Italy.}

\keywords{Overcomplete sequence, overcomplete set, WLD Banach space, Marku\v{s}evi\v{c} basis}
\subjclass[2010]{46B20, 46B50 (primary), and 46A35, 46B26 (secondary).}
\date{\today}

\begin{abstract} We introduce and study the notion of overcomplete set in a Banach space, that subsumes and extends the classical concept of overcomplete sequence in a (separable) Banach space. We give existence and non-existence results of overcomplete sets for a wide class of (non-separable) Banach spaces and we study to which extent properties of overcomplete sequences are retained by every overcomplete set.
\end{abstract}
\maketitle

\section{Introduction}
The notion of \emph{overcomplete} sequence originated in a 1958 paper by Victor Klee \cite{Klee} (let us also refer to \cite[p.~283]{BePe} for some historical comments), where the author constructs, in every separable Banach space, a sequence whose every subsequence is linearly dense. This well-known construction is such a sparkling gem that we shall briefly outline it here. Given a real scalar $\lambda\in(0,1)$, consider the geometric vector
$$g_\lambda:=(1,\lambda,\lambda^2,\lambda^3,\dots)\in\ell_1.$$
A consequence of the Hahn-Banach theorem and the identity principle for analytic functions is that the set $\{g_\lambda\colon \lambda\in J\}$ is linearly dense in $\ell_1$, whenever $J$ is an infinite subset of $(0,1/2)$. Indeed, if $x=(x(j))_{j<\omega}\in\ell_\infty$ is a functional that vanishes on each $g_\lambda$ ($\lambda\in J$), where $J\subseteq (0,1/2)$ is infinite, then $\sum_{j=0}^\infty x(j)\cdot \lambda^j=0$, for every $\lambda\in J$. In other words, the analytic function $\lambda\mapsto \sum_{j=0}^\infty x(j)\cdot \lambda^j$ has infinitely many zeros on $[0,1/2]$, hence it is identically equal to zero, by the identity principle. It follows that $x=0$, whence $\{g_\lambda\colon \lambda\in J\}$ is linearly dense, as required. Consequently, overcomplete sequences exist in the Banach space $\ell_1$. 

In order to construct an overcomplete sequence in an arbitrary separable Banach space, it is then sufficient to remind that every separable Banach space is a quotient of $\ell_1$ and observe that overcomplete sequences are preserved under quotient maps. (For more details the reader might consult \cite[p.~58]{SingerII} or \cite[p.~113]{Milman}.)\smallskip

Such sequences have been studied by several authors \cite{CFP, T1,T2,T3,T4}, in particular in connection with the existence of basic sequences. It is, indeed, plain that an overcomplete sequence cannot admit any basic subsequence (or, more generally, any minimal subsequence). On the other hand, Terenzi \cite{T4} found a surprising dichotomy implying that, in a sense, being overcomplete and being basic are complementary notions.\smallskip

More recently \cite{FonfZanco}, the dual notion of \emph{overtotal} sequence has been introduced and both notions have been generalised, leading to the definitions of \emph{almost overcomplete} and \emph{almost overtotal} sequences (see also \cite{FSTZ}). Perhaps the main result in \cite{FonfZanco}, that will also play a cardinal r\^{o}le in our note, is the fact that every bounded almost overcomplete sequence, as well as every bounded almost overtotal sequence, is relatively compact. In the same article, such a property has been used to provide a short and unified approach to some spaceability results, \cite{CSS,EGSS}.\smallskip

Our purpose in the present note is to push the theory of overcomplete sequences to its natural non-separable counterpart, by means of the introduction and study of the concept of overcomplete subset of a Banach space. Let us therefore formulate the following definition.
\begin{definition} A subset $S$ of a Banach space $X$, with $|S|=\dens\,X$, is said to be \emph{overcomplete} if every subset $\Lambda$ of $S$, with $|\Lambda|=|S|$, is linearly dense in $X$.
\end{definition}
It is, for example, plain that, in the separable context, any injective enumeration of an overcomplete set is an overcomplete sequence and, vice versa, the range of an overcomplete sequence is an overcomplete set.\smallskip

As we shall see, unlike the separable case, overcomplete sets may fail to exist in non-separable Banach spaces, particularly in Banach spaces with `big' density. As a sample of this phenomenon, to be investigated in Section \ref{S: existence}, let us offer here the following result.
\begin{theoremA} Let $X$ be a WLD Banach space. Then:
\begin{romanenumerate}
    \item {\rm (CH)} if $\dens\,X=\omega_1$, $X$ contains an overcomplete set;
    \item if $\dens\,X\geq\omega_2$, $X$ contains no overcomplete set.
\end{romanenumerate}
On the other hand, $\ell_1(\omega_1)$ contains no overcomplete set.
\end{theoremA}

The proof of Theorem A will be given in Section \ref{S: existence}, where we also obtain some more precise results. In particular, we prove, under CH, that every Banach space $X$ with $\dens\,X=\dens\,X^*=\omega_1$ contains overcomplete sets; this result covers, \emph{e.g.}, the cases of WLD spaces of density $\omega_1$ and of $C(K)$-spaces, where $K$ is a scattered compact with weight $\omega_1$. On the other hand, we prove (in ZFC) that a Banach space $X$ with $\kappa :=\dens\,X$ does not admit overcomplete sets, provided $\cf(\kappa)\geq\mathfrak{c}^+$, or $\kappa\geq\omega_2$ and $X$ admits a fundamental biorthogonal system.

As it is perhaps clear, our results leave a large area for further investigation and suggest several natural questions. Let us explicitly record some of them here. Since every our existence result is obtained under CH, a natural question is whether there exists, in ZFC, a non-separable Banach space admitting an overcomplete set. Moreover, we could ask the following:

\begin{problem}
Assume CH.
\begin{romanenumerate}
\item Does $\ell_{\infty}$ admit an overcomplete set? (Compare with Corollary \ref{cor: negCH ell_infty})
\item Is there a Banach space of density $\omega_\omega$ and with an overcomplete set?
\end{romanenumerate}
\end{problem}

Subsequently, in Section \ref{S: properties}, we initiate the study of the properties of overcomplete sets in Banach spaces. Our main concern relies in detecting the correct non-separable counterpart of the result that bounded overcomplete sequences are relatively compact. Certainly, overcomplete sets in non-separable Banach spaces fail to be relatively compact, since compact sets are separable; consequently, a direct extension of the result from \cite{FonfZanco} is not available. Analogously, the existence of a relatively weakly compact overcomplete set in a Banach space $X$ implies that $X$ is WCG. This motivates the following strengthening of Theorem A(i), in the class of WCG Banach spaces.
\begin{theoremB}[CH] Every WCG Banach space $X$ with $\dens\,X=\omega_1$ contains a relatively weakly compact overcomplete set.
\end{theoremB}

\section{Preliminaries}
Our notation concerning Banach spaces, which we only consider over the real field, is standard, as in \cite{ak, FHHMZ}. Let us just mention that by \emph{hyperplane} in a Banach space we understand a closed linear subspace of codimension one, namely, the kernel of a non-zero, bounded linear functional. We shall briefly remind here some basic definitions concerning non-separable Banach spaces, that can be found, \emph{e.g.}, in \cite{HMVZ}.

A \emph{biorthogonal system} in a Banach space $X$ is a system $\{x_\alpha;f_\alpha\} _{\alpha\in\Gamma}\subseteq X\times X^*$ such that $\langle f_\alpha, x_\beta\rangle =\delta_{\alpha,\beta}$, whenever $\alpha,\beta\in\Gamma$. A biorthogonal system  $\{x_\alpha;f_\alpha\}_{\alpha\in\Gamma}$ is \emph{fundamental} if
 $\overline{\rm span}\{x_\alpha\}_{\alpha\in\Gamma}=X;$
it is a \emph{Marku\v{s}evi\v{c} basis} (\emph{M-basis}, for short) if
$$\overline{\rm span}\{x_\alpha\}_{\alpha\in\Gamma}=X\qquad\text{and}\qquad \overline{\rm span}^{w^*}\{f_\alpha\}_{\alpha\in\Gamma}=X^*.$$

Once an M-basis $\{x_\alpha;f_\alpha\} _{\alpha\in\Gamma}$ is present in a Banach space $X$, there is a natural notion of \emph{support} for a vector in $X$, as follows:
$${\rm supp}\,(x):=\{\alpha\in\Gamma\colon \langle f_\alpha,x\rangle\neq 0\}.$$
It is well-known that ${\rm supp}\,(x)$ is a countable subset of $\Gamma$, for every $x\in X$. More generally, one shows that any M-basis induces a linear continuous injection $T$ of $X$ into $c_0(\Gamma)$, via the map
$$x\mapsto T(x):=\left(\left\langle\frac{f_\alpha}{\|f_\alpha\|}, x\right\rangle\right)_{\alpha\in\Gamma}.$$
The definition of support can also be given when $\{x_\alpha;f_\alpha\} _{\alpha\in\Gamma}$ is merely a fundamental biorthogonal system and it is still true that the support of every vector in $X$ is countable. What ceases to be true is that $\supp(x)\neq\emptyset$ for $x\neq0$; accordingly, the map $T$ defined above might fail to be injective.

Analogously, the \emph{support} of a functional $x^*\in X^*$ is the set ${\rm supp}\,(x^*) :=\{ \alpha\in\Gamma\colon \langle x^*,x_\alpha\rangle\neq 0\}$ (which, in general, may fail to be a countable set).\smallskip

Let us then recall that a Banach space $X$ is \emph{weakly Lindel\"of determined} (WLD, for short) if there exist a set $\Gamma$ and a bounded linear one-to-one operator $T\colon X^*\to \ell_{\infty}^c(\Gamma)$ that is weak$^*$-to-pointwise continuous, see \cite{ArgMer}. Here, $\ell_{\infty}^{c}(\Gamma)$ stands for the Banach space of all real valued, bounded functions on $\Gamma$ with countable support. Moreover, we recall that if $X$ is a WLD Banach space with $\dens X=\kappa$, then, in the definition of the operator $T$, $\Gamma$ can be replaced by $\kappa$. WLD Banach spaces admit a simple characterisation in terms of M-bases, as follows. A Banach space $X$ is WLD if and only if it admits an M-basis such that the support of every functional in $X^*$ is countable; in which case, each M-basis has such property (see, \emph{e.g.}, \cite[Theorem 5.37]{HMVZ}).\smallskip

An important subclass of the class of WLD spaces is the class of weakly compactly generated Banach spaces. A Banach space is \emph{weakly compactly generated} (WCG, for short) if it contains a linearly dense, weakly compact subset. Likewise for WLD spaces, WCG Banach spaces can be characterised by the existence of M-bases with specific properties: a Banach space $X$ is WCG if and only if it admits a \emph{weakly compact} M-basis $\{x_{\alpha};f_{\alpha}\}_{\alpha\in\Gamma}$, \emph{i.e.}, such that $\{x_{\alpha}\}_{\alpha\in\Gamma}\cup \{0\}$ is a weakly compact subset of $X$, \cite{AmirLin} (see also \cite[Theorem 13.16]{FHHMZ}).\smallskip

Finally, let us mention that our notation concerning set theory follows, \emph{e.g.}, \cite{Jech}, where all notions here undefined can be found. Here, we just recall the statement of Hajnal's theorem on free sets (\cite{Haj61}, see \emph{e.g.}, \cite[\S 44]{EHMR} or \cite[\S 3.1]{Wil77}).
\begin{theorem}[Hajnal, \cite{Haj61}]\label{Hajnal th} Let $\kappa\geq\omega_2$ be a cardinal number. Then for every function $f\colon\kappa\to [\kappa]^{\leq\omega}$ there exists  $H\subseteq\kappa$ with $|H|=\kappa$ such that $f(x)\setminus\{x\}$ is disjoint from $H$, whenever $x\in H$.
\end{theorem}

\section{Existence results}\label{S: existence}
In this section, we shall present results concerning the existence or non-existence of overcomplete sets in non-separable Banach spaces; in particular, these results prove Theorem A. Loosely speaking, overcomplete sets do not exist in `large' Banach spaces; more precisely, we will show that a Banach space contains no overcomplete set if:
\begin{romanenumerate}
    \item $\cf(\dens X)\geq \mathfrak{c}^+$, or
    \item $\dens X\geq \omega_2$, provided $X$ admits a fundamental biorthogonal system.
\end{romanenumerate}
On the other hand, under the assumption of the continuum hypothesis, we will show that every `small' Banach space $X$ (\emph{i.e.}, $\dens X=\dens X^*=\omega_1$) contains overcomplete sets. Finally, we will show that the assumption on the density character of the dual space cannot be dropped. Indeed we will show, in ZFC, that $\ell_1(\omega_1)$ does not contain overcomplete sets.

\begin{theorem}[CH]\label{T: existence under CH}
Let $X$ be a Banach space with ${\rm dens}\,X={\rm dens}\,X^*=\omega_1$ and let $C\subseteq X$. If $C$ cannot be covered by countably many hyperplanes, then it contains an overcomplete set for $X$.
\end{theorem}

Since a Banach space cannot be covered by countably many hyperplanes, in light of the Baire category theorem, we have the following particular case.
\begin{corollary}[CH]\label{C: existence under CH} Every Banach space $X$ with ${\rm dens}\,X={\rm dens}\,X^*=\omega_1$ contains an overcomplete set.
\end{corollary}

\begin{proof}[Proof of Theorem \ref{T: existence under CH}] We first notice that $|X^*|=\mathfrak{c}$: indeed, since every element of $X^*$ is a limit of a sequence contained in some dense subset $D$ of $X^*$ with $|D|=\mathfrak{c}$, we have $|X^*|\leq |D^\omega|=\mathfrak{c}^\omega=\mathfrak{c}$. Consequently, the collection of hyperplanes of $X$ has cardinality $\mathfrak{c}$ and the assumption of CH allows us to well order it in an injective $\omega_1$-sequence $(H_\alpha)_{\alpha<\omega_1}$. By transfinite induction, we now build an injective $\omega_1$-sequence of elements $(x_\beta)_{\beta<\omega_1}\subseteq C$ such that
$$x_\beta\notin H_\alpha\qquad \text{whenever }\alpha<\beta.$$
Indeed, let us assume to have already built $(x_\beta)_{\beta<\gamma}$, for some $\gamma<\omega_1$. Then,
$$\bigcup_{\alpha<\gamma}H_\alpha\cup \bigcup_{\beta<\gamma}\{x_\beta\}$$
does not cover $C$. Thus, every vector $x_\gamma\in C$ outside such a set is as desired.

Finally, we show that $(x_\beta)_{\beta<\omega_1}$ is overcomplete in $X$. Indeed, if not, there would be an uncountable subset $\Lambda$ of $\omega_1$ such that the closed linear span of $(x_\beta)_{\beta\in\Lambda}$ is a proper subspace of $X$, hence there would be a non-zero linear functional that vanishes on each $x_\beta$ ($\beta\in\Lambda$), \emph{i.e.}, its kernel contains uncountably many $x_\beta$'s. This is, however, a contradiction, because our construction assures that every hyperplane contains at most countably many elements of the sequence $(x_\beta)_{\beta<\omega_1}$.
\end{proof}

\begin{remark} The previous theorem actually characterises, under CH, sets that contain overcomplete subsets for $X$. Indeed, it follows from a conjunction of it and Proposition \ref{Prop: omega hyperplanes} that, under the assumptions of Theorem \ref{T: existence under CH}, a set $C$ contains an overcomplete subset for $X$ if and only if $C$ is not covered by countably many hyperplanes.
\end{remark}

The previous result allows us to ensure the existence, under the continuum hypothesis, of overcomplete sets in a wide class of Banach spaces.

\begin{corollary}[CH] Let $X$ be a WLD Banach space with $\dens X=\omega_1$; then X contains an overcomplete subset.
\end{corollary}

\begin{proof} The assertion follows by Corollary \ref{C: existence under CH} and the fact that, under the continuum hypothesis, the cardinality of $\ell_{\infty}^{c}(\omega_1)$ is equal to $\omega_1$, whence $|X^*|=\omega_1$ as well.
\end{proof}

Let us remark that small WLD Banach spaces are not the unique Banach spaces containing overcomplete subsets. Indeed, $C([0,\omega_1])$, a typical example of a Banach space that is not WLD, contains overcomplete subsets, due to the next result.

\begin{corollary}[CH] Let $K$ be a scattered compact space of weight $\omega_1$. Then, $C(K)$ has an overcomplete subset.
\end{corollary}

\begin{proof} Since $K$ has weight $\omega_1$, $\dens C(K)=\omega_1$ too. Moreover, by \cite[Proposition 17.10]{Kopp}, we have $|K|=\omega_1$, hence, by \cite[Theorem 6]{Rudin1}, the dual space of $C(K)$ is isometric to $\ell_1(\omega_1)$. Therefore, the assertion follows by Corollary \ref{C: existence under CH}.
\end{proof}

In the forthcoming results, we shall show that Banach spaces with large size do not contain overcomplete subsets. Let us stress the fact that each of this results is achieved inside ZFC, without any further axiom. In particular, Theorem A(ii) follows from the next result.

\begin{theorem}\label{th: omega2 Hajnal} Let $X$ be a Banach spaces with ${\rm dens}\,X\geq\omega_2$. Suppose that $X$ has a fundamental biorthogonal system; then $X$ does not contain overcomplete subsets.
\end{theorem}
\begin{proof} 
Let $\kappa:=\dens X$ and $\{x_{\alpha};f_{\alpha}\}_{\alpha<\kappa}$ be a fundamental biorthogonal system for $X$. We recall that the support $\supp \,(x)$ of every element $x\in X$ is a countable subset of $\kappa$. Suppose, by contradiction, that an overcomplete subset $S$ in $X$ does exist and order it in a transfinite sequence $S=(y_{\alpha})_{\alpha<\kappa}$. Let $g\colon \kappa\to 2^{\kappa}$ defined by $g(\alpha)=\supp (y_\alpha)$. By Hajnal's theorem (Theorem \ref{Hajnal th}), there exists a set $H\subseteq \kappa$, with $|H|=\kappa$, such that $(g(\alpha)\setminus\{ \alpha\})\cap H=\emptyset$ for every $\alpha \in H$. In particular, fixed $\gamma\in H$, we deduce that $\gamma \notin \supp (y_{\alpha})$, whenever $\alpha \in H\setminus \{\gamma\}$. Hence, we obtain that $(y_{\alpha})_{\alpha\in H\setminus \{\gamma\}}\subseteq \ker f_{\gamma}$, which contradicts the overcompleteness of $S$.
\end{proof}

Similarly as in the separable case, the existence of overcomplete sets passes from a Banach space to each its quotient of the same density character. Thus, Theorem \ref{th: omega2 Hajnal} implies that a Banach space $X$ fails to have overcomplete sets if there exists a quotient $Y$ of $X$ with $\dens\,Y=\dens\,X\geq \omega_2$ and such that $Y$ admits a fundamental biorthogonal system. In particular, this applies to $\ell_\infty$, since it quotients onto $\ell_2 (\mathfrak{c})$ and $\dens\,\ell_\infty=\dens\,\ell_2 (\mathfrak{c})=\mathfrak{c}$. Therefore, we have the following consequences of Theorem \ref{th: omega2 Hajnal}. (We are indebted to W.B.~Johnson for pointing out this to us. We also thank W.~Kubi\'s for indicating us a different proof of Corollary \ref{cor: negCH ell_infty}, based on \cite{Kopp2}.)

\begin{corollary} Let $X$ be a Banach space with $\dens\,X\geq\omega_2$. If $X$ admits a WCG quotient $Y$ with $\dens\,Y=\dens\,X$, then $X$ admits no overcomplete set.
\end{corollary}
\begin{corollary}[$\neg$ CH]\label{cor: negCH ell_infty} $\ell_\infty$ has no overcomplete set.
\end{corollary}

In the case where $\dens X$ has big cofinality, we can give a simpler proof of the previous theorem (even in absence of a fundamental biorthogonal system), by means of the following result.
\begin{theorem} Let $X$ be a Banach space with {\rm dens}\,$X=\kappa$. If $\cf(\kappa) \geq\mathfrak{c}^+$, $X$ contains no overcomplete set.
\end{theorem}

\begin{proof} Let us start by recalling that every Banach space is union of $\mathfrak{c}$ hyperplanes. Indeed, if $x^*,y^*\in X^*$ are linearly independent functionals, then
\begin{equation*}X=\bigcup_{(\alpha,\beta)\in\R^2\setminus\{0\}} \ker(\alpha x^*+\beta y^*).
\end{equation*}

Now, let $A$ be a subset of $X$ such that $|A|=\kappa$. Since $A$ is contained in a union of $\mathfrak{c}$ hyperplanes and $\cf(\kappa)\geq \mathfrak{c}^+$, there exists a hyperplane that contains $\kappa$ vectors from $A$. Consequently, $A$ cannot be overcomplete.
\end{proof}

\begin{proposition}\label{Prop: omega hyperplanes}
Let $X$ be a Banach space with $\dens X=\kappa$. If $D\subseteq X$ is covered by less than $\cf(\kappa)$ many hyperplanes, then $D$ does not contain overcomplete sets.
\end{proposition}

\begin{proof}
Suppose by contradiction that $D$ contains an overcomplete set, say $A$. Let $(H_{\alpha})_{\alpha<\lambda}$ be a family of hyperplanes of $X$ such that $D\subseteq \bigcup_{\alpha<\lambda}H_\alpha$, where $\lambda<\cf(\kappa)$. Since $|A|=\kappa$ and $\cf(\kappa)>\lambda$, there exists $\alpha<\lambda$ such that $|A\cap H_\alpha|=\kappa$, a contradiction.
\end{proof}

\begin{remark} The previous result shows that many incomplete non-separable normed spaces fail to contain overcomplete subsets. Indeed, let $X$ be a Banach space that admits an M-basis $\{x_{\alpha};f_{\alpha}\}_{\alpha<\kappa}$ and with $\cf\kappa \geq\omega_1$. Then, the (incomplete) normed space $Y:={\rm span}\{x_\alpha\}_{\alpha<\kappa} $ does not contain overcomplete sets, since $Y\subseteq \bigcup_{n<\omega}\ker f_n$ (here, we apply the previous proposition with $D=Y$). Actually, the assumption $\cf\kappa \geq\omega_1$ can be relaxed to $\kappa\geq\omega_1$, by using a similar argument as in Theorem \ref{th: omega2 Hajnal}. The situation is different in the separable context, as we shall see in Section \ref{S: normed spaces}.
\end{remark}

\begin{theorem} The Banach space $\ell_1(\Gamma)$ does not contain any overcomplete set, whenever $\Gamma$ is uncountable.
\end{theorem}

\begin{proof} In the case where $|\Gamma|\geq\omega_2$, the result is consequence of Theorem \ref{th: omega2 Hajnal}; therefore, we can assume that $\Gamma=\omega_1$. Let us then assume, by contradiction, that $S$ is an overcomplete set in $\ell_1(\omega_1)$; clearly, we can additionally assume, without loss of generality, that every element in $S$ is a unit vector.

For each $\alpha<\omega_1$, consider the quantity
$$N_\alpha:=\inf\left\{\|x\flag_{[0,\alpha)}\|\colon x\in S \right\},$$
where, for a vector $x=(x(\gamma))_{\gamma<\omega_1}\in \ell_1(\omega_1)$, we understand $x\flag_{[0,\alpha)}:=(x(\gamma)\cdot 1_{[0,\alpha)}(\gamma))_{\gamma<\omega_1}$. Observe that, when $\alpha<\beta<\omega_1$, we have $N_\alpha\leq N_\beta$; consequently, the function $\alpha\mapsto N_\alpha$ is eventually constant and we can select $\alpha_0<\omega_1$ such that
$$N_\alpha = N_{\alpha_0}\;\text{ whenever }\; \alpha_0\leq \alpha<\omega_1.$$

Set $N:=N_{\alpha_0}$ and note that, evidently, $N\in[0,1]$. We first claim that, actually, $N<1$. Indeed, if it were $N=1$, then $\|x\flag_{[0,\alpha_0)}\|=1$, for each $x\in S$, whence ${\rm supp}\,(x)\subseteq[0,\alpha_0)$ for every $x\in S$. This is, of course, impossible, since the set $S$ is (over)complete. Let us now fix $\e>0$ such that $1-N-2\e\geq \frac{1-N}{2}$; by definition, we then have the following property:
\begin{equation}\label{Eq: sliding}
    \text{for every}\; \beta\geq\alpha_0\; \text{there exists}\; x\in S  \;\text{such that}\; \|x\flag_{[0,\beta)}\|\leq N+\e.
\end{equation}

This property allows us to run a sliding hump argument, by transfinite induction. Select arbitrarily a vector $x_0\in S$ such that $\|x_0\flag_{[0,\alpha_0)}\|\leq N+\e$; then, find an ordinal $\alpha_1<\omega_1$ such that $\supp\,(x_0)<\alpha_1$ and appeal to (\ref{Eq: sliding}) to find $x_1\in S$ such that $\|x_1\flag_{[0,\alpha_1)}\|\leq N+\e$. By transfinite induction, we then find a long sequence $(x_\gamma)_{\gamma<\omega_1}\subseteq S$ and an increasing long sequence $(\alpha_\gamma)_{\gamma<\omega_1}\subseteq\omega_1$ such that, for every $\gamma<\omega_1$,
\begin{romanenumerate}
    \item $\|x_\gamma\flag_{[0,\alpha_\gamma)}\|\leq N+\e$;
    \item $\supp (x_\beta) <\alpha_\gamma$, whenever $\beta<\gamma$.
\end{romanenumerate}
Indeed, assuming to have already built vectors $(x_\beta)_ {\beta<\gamma}$ and ordinals $(\alpha_\beta)_{\beta<\gamma}$, for some $\gamma<\omega_1$, we just select $\alpha_\gamma$ to satisfy (ii) and then (\ref{Eq: sliding}) yields us the desired vector $x_\gamma$.

Let us then observe that the vectors $(x_\gamma)_{\gamma<\omega_1}$ have the following additional properties:
\begin{romanenumerate}\setcounter{smallromans}{2}
    \item $\|x_\gamma\flag_{[\alpha_\gamma,\omega_1)}\|\geq1-N-\e$;
    \item $\|x_\gamma\flag_{[\alpha_0,\alpha_\gamma)}\|\leq\e$.
\end{romanenumerate}
Indeed, (iii) follows from (i), as $x_\gamma$ is a unit vector; (iv) follows from (i) either and the fact that $\|x_\gamma\flag_{[0,\alpha_0)}\|\geq N$.\smallskip

We are now in position to prove that the vectors $(x_\gamma)_{\gamma<\omega_1}$ are indeed a long basic sequence, equivalent to the canonical $\ell_1(\omega_1)$ basis. In order to prove such claim, consider the vectors
$$y_\gamma:=x_\gamma - x_\gamma\flag_{[\alpha_0,\alpha_\gamma)}= x_\gamma\flag_{[0,\alpha_0)}+x_\gamma\flag_{[\alpha_\gamma, \omega_1)};$$
then $\|x_\gamma-y_\gamma\|\leq\e$, according to (iv). Moreover, the vectors
$$y_\gamma\flag_{[\alpha_0,\omega_1)}= x_\gamma\flag_ {[\alpha_\gamma,\omega_1)}$$
are disjointly supported, due to (ii). Therefore, for every choice of $n<\omega$, indices $\gamma_0,\dots,\gamma_n<\omega_1$ and scalars $a_0,\dots,a_n$, we have
$$\left\|\sum_{j=0}^n a_jx_{\gamma_j}\right\|\geq \left\|\sum_{j=0}^n a_jy_{\gamma_j}\right\|- \e\cdot\sum_{j=0}^n |a_j|\geq \left\|\sum_{j=0}^n a_j\cdot y_{\gamma_j} \flag_{[\alpha_0,\omega_1)} \right\|- \e\cdot\sum_{j=0}^n |a_j|$$
$$=\sum_{j=0}^n|a_j| \left\|y_{\gamma_j} \flag_{[\alpha_0, \omega_1)}\right\| - \e\cdot\sum_{j=0}^n |a_j|\geq (1-N-2\e) \cdot\sum_{j=0}^n |a_j|\geq\frac{1-N}{2}\cdot\sum_{j=0}^n |a_j|,$$
where we used the disjointness of supports and (iii). The upper estimate being direct consequence of the triangle inequality, our claim is proved. Moreover, the argument is also concluded, since the basic sequence $(x_\gamma)_{\gamma<\omega_1}\subseteq S$ is not overcomplete, thereby giving the desired contradiction.
\end{proof}

\begin{remark} It is a standard result that every Banach space of density $\kappa$ is a quotient of $\ell_1(\kappa)$. Moreover, it is plain that the existence of an overcomplete set passes from a Banach space to each its quotient of the same density character. Therefore, a natural approach to build overcomplete sets in non-separable Banach spaces would be to construct them in the spaces $\ell_1(\Gamma)$. Our previous result, however, implies in particular that this approach fails to work in the non-separable framework.
\end{remark}

\section{Properties of overcomplete sets}\label{S: properties}
In the present section, we shall study some topological properties of overcomplete sets in (non-separable) Banach spaces, focusing mostly on density and compactness results. Many of our results in the section can be interpreted as counterparts (or absence of a counterpart) to the result by Fonf and Zanco \cite{FonfZanco} that bounded overcomplete sequences in separable Banach spaces are relatively compact.

\subsection{Density}
It follows from the above-mentioned compactness result that every overcomplete sequence in an infinite-dimensional, separable Banach space is necessarily a nowhere dense set; on the other hand, it is a simple folklore result that every finite-dimensional Banach space contains overcomplete sequences that are dense sets. Let us record such a fact here; we also include its simple proof, for convenience of the reader.

\begin{fact} Every finite-dimensional Banach space contains a dense overcomplete sequence.
\end{fact}
\begin{proof} A simple way to build an overcomplete sequence in a finite-dimensional Banach space $X$ is to construct a sequence $(x_j)_{j<\omega}$ in $X$ such that every $d$ many terms extracted from the sequence are linearly independent, where $d= {\rm dim}\,X$. Such a sequence can be built by induction: once the first $d+k$ terms $x_0,\dots,x_{d+k-1}$ have been selected, every $d-1$ of the vectors $x_0,\dots,x_{d+k-1}$ are contained in exactly one hyperplane. The desired vector $x_{d+k}$ is then any vector that does not belong to any of those finitely many hyperplanes.

Only a small modification is needed in order to achieve the density of $(x_j)_{j<\omega}$. Indeed, let $(U_j)_{j<\omega}$ be an enumeration of a basis for the topology of $X$, where each $U_j$ is not empty. Since non-empty open sets cannot be covered by finitely many hyperplanes, we can, at each step, choose $x_j\in U_j$, which assures that $(x_j)_{j<\omega}$ is dense in $X$.
\end{proof}

A simple combination of the above argument and the proof of the main existence result (Theorem \ref{T: existence under CH}) allows us to extend the above fact to the non-separable context. We thus have the following result.

\begin{theorem}[CH] Let $X$ be a Banach space with ${\rm dens}\,X={\rm dens}\,X^*=\omega_1$. Then, $X$ contains a dense overcomplete set.
\end{theorem}
\begin{proof} Let $(U_{\alpha})_{\alpha<\omega_1}$ be an enumeration of a basis for the topology of $X$, where each open set $U_\alpha$ is not empty. As in Theorem \ref{T: existence under CH}, we enumerate the hyperplanes of $X$ as $(H_{\alpha})_{\alpha <\omega_1}$. Since open subsets of Banach spaces can not be covered by countably many hyperplanes, by the Baire category theorem, we can pick, for each $\alpha<\omega_1$, a vector $x_{\alpha}\in U_{\alpha}$ such that $x_{\alpha}\notin H_{\beta}$, whenever $\beta<\alpha$. Therefore, we obtain a long injective sequence $(x_{\alpha})_{\alpha<\omega_1}$ with $x_\alpha\in U_\alpha$ ($\alpha<\omega_1$) and such that every hyperplane contains at most countably many elements of $(x_{\alpha})_{\alpha<\omega_1}$, as desired.
\end{proof}

On the other hand, it is also possible to adapt the proof of Theorem \ref{T: existence under CH} to obtain bounded overcomplete sets that are separated sets, as we do below. Of course, such a phenomenon is impossible in the separable case.
\begin{proposition}[CH] Let $X$ be a Banach space with ${\rm dens}\,X={\rm dens}\,X^*=\omega_1$. Then, for every $\e>0$, $B_X$ contains a $(1-\e)$-separated overcomplete set.
\end{proposition}
\begin{proof} Arguing as in the proof of Theorem \ref{T: existence under CH}, let $(H_\alpha)_{\alpha<\omega_1}$ be an enumeration of the hyperplanes of $X$. We then build a $(1-\e)$-separated long sequence $(x_\alpha)_{\alpha <\omega_1}$ such that $x_\alpha \notin H_\beta$ ($\beta<\alpha<\omega_1$). Indeed, assuming to have already built $(x_\alpha)_{\alpha <\gamma}$, for some $\gamma<\omega_1$, $Y:=\overline{\rm span}\{x_\alpha\}_{\alpha <\gamma}$ is a proper subspace of $X$ and therefore, by Riesz' lemma,
$$U:=\{x\in X\colon \|x\|<1,\; {\rm dist}(x,Y)>1-\e\}$$
is a non-empty open subset of $X$. We can then pick $x_\gamma\in U\setminus \cup_{\alpha<\gamma}H_\alpha$, which concludes the transfinite induction and, thereby, the proof.
\end{proof}

One might wonder if the above result can be improved in order to produce overcomplete sets with no cluster point also in the weak topology. It is, however, elementary to realise that this is, in general, not possible. Indeed, every WLD Banach space is Lindel\"of in the weak topology (see, \emph{e.g.}, \cite[Theorem 17.1]{KKLP}) and it is clear that every uncountable subset of a Lindel\"of topological space needs to have a cluster point. On the other hand, we do not know what happens for Banach spaces that are not WLD. For example, is it possible to find an overcomplete set in $C([0,\omega_1])$ with no weak cluster point? Moreover, we do not know if it is possible, say under CH, to construct overcomplete sets that are discrete in the weak topology.

\subsection{Compactness} 
As we already observed in the introduction, no overcomplete set in a non-separable Banach space can be relatively compact and, therefore, there is no non-separable extension to the result in \cite{FonfZanco}. (The assertion `every bounded, overcomplete set in a non-separable Banach space is relatively weakly compact' could be consistently true, since it is true precisely in every model of ZFC where non-separable Banach spaces have no overcomplete set.)

Moreover, the existence of a relatively weakly compact overcomplete set in $X$ implies that $X$ is WCG and, each bounded, overcomplete set in a reflexive Banach space is relatively weakly compact. The results of this section yield, under CH, a converse to the last two assertions and, thereby, provide a characterisation of WCG and reflexive Banach spaces in terms of overcomplete sets (see Corollary \ref{C: WCG and refl}).

\begin{proposition}
Let $X$ be a non-separable non-reflexive Banach space. If $X$ contains an overcomplete subset, then $X$ contains a bounded overcomplete subset which is not relatively weakly compact.
\end{proposition}

\begin{proof} Since the unit ball of $X$ is not relatively weakly compact, we can pick a sequence $(x_n)_{n<\omega} \subseteq B_{X}$ which has no weakly convergent subsequences. By assumption, there exists a bounded overcomplete subset $(z_{\alpha})_{\alpha<\omega_1}$ in $X$. The set $Z=(x_n)_{n<\omega}\cup(z_{\alpha})_{\alpha<\omega_1}$ is then a bounded overcomplete subset of $X$ which is not relatively weakly compact. 
\end{proof}

\begin{theorem}[CH] Let $X$ be a WCG Banach space with ${\rm dens}\,X=\omega_1$; then, $X$ contains a relatively weakly compact overcomplete set.
\end{theorem}

\begin{proof} Since $X$ is a WCG Banach space, we can pick a  linearly dense, weakly compact subset $K$ of $X$. By Krein's theorem, the set $C:=\overline{\rm conv}(K)$ is also weakly compact. Therefore, it is sufficient to find an overcomplete set for $X$ that is contained in $C$.
In light of Theorem \ref{T: existence under CH}, this amounts to proving the following claim.

\begin{claim}\label{Claim WCG} The set $C$ cannot be covered by countably many hyperplanes.
\end{claim}
\begin{proof}[Proof of Claim \ref{Claim WCG}] Assume, by contradiction, that there exists a sequence $(H_j)_{j<\omega}$ of hyperplanes such that $C= \cup_{j<\omega}(C \cap H_j)$. By the Baire category theorem, there exists $j_0<\omega$ such that $C \cap H_{j_0}$ has nonempty interior in $C$. This is, however, impossible. Indeed, let $x$ be an interior point in $C \cap H_{j_0}$ and let $y\in C\setminus H_{j_0}$ (such $y$ exists as $C$ is linearly dense, hence it cannot be contained in a hyperplane). Then, the sequence $((1-2^{-n})x+2^{-n}y)_{n<\omega}$ belongs to $C\setminus H_{j_0}$ and it converges to $x$, contrary to the assumption that $x$ was an interior point of $C \cap H_{j_0}$.
\end{proof}\end{proof}

As direct consequence of the results of this section, we obtain, still under CH, the following characterisation of WCG and reflexive Banach spaces of density $\omega_1$ by means of overcomplete sets.
\begin{corollary}[CH]\label{C: WCG and refl} Let $X$ be a Banach space with ${\rm dens}\,X={\rm dens}\,X^*=\omega_1$. Then
\begin{romanenumerate}
    \item $X$ is WCG if and only if it contains a relatively weakly compact overcomplete set;
    \item $X$ is reflexive if and only if every bounded overcomplete set is relatively weakly compact.
\end{romanenumerate}
\end{corollary}

\section{Overcomplete sequences in normed spaces}\label{S: normed spaces}
In this final section, we shall return to the separable framework and present some further results on overcomplete sequences; in particular, we shall focus on existence and properties of overcomplete sequences in separable normed spaces.

It is apparent that Klee's argument requires the completeness of the space in order to assure the convergence of the geometric series. On the other hand, there exists an alternative way to build overcomplete sequences that only involves `finitely supported' vectors and, as such, is valid for every separable normed space (see, \emph{e.g.}, \cite[Theorem 2.1.2]{GurariyLusky} or \cite[p. 59]{SingerII}, where it is given credit to Helmut Bra\ss, \cite{Brass}). Therefore, we have the following result.
\begin{proposition}[Bra\ss, \cite{Brass}] Every separable normed space contains an overcomplete sequence.
\end{proposition}

It is then natural to ask which properties overcomplete sequences in normed spaces need to have. In particular, it seems reasonable to speculate that the result by Fonf and Zanco \cite{FonfZanco} on the relative compactness of bounded overcomplete sequences depends on the completeness of the space under consideration. (This question was asked to us by Prof.~Zanco \cite{Zanco}.) We show that this intuition is correct, by proving that the result by Fonf and Zanco is false in every incomplete, separable normed space.
\begin{theorem} Every separable, incomplete normed space contains a bounded overcomplete sequence that is not relatively compact.
\end{theorem}

\begin{proof}
Let $\widehat{X}$ denote the completion of $X$ and let $(x_n)_{n<\omega}$ be a normalised complete sequence in $X$ (which is, evidently, also complete for $\widehat{X}$). Let $y\in \widehat{X}\setminus X$ and pick a sequence $(y_k)_{k<\omega} \subseteq X$ convergent to $y$ in such a way that $\|y_k - y\|<1/k!$. Set
\begin{equation}\label{eq: g_k overcomplete for normed}
   g_k= y_k + \sum_{n=0}^{k}(n+2)^{-k}x_n \in X \qquad(k<\omega).
\end{equation}
Let us first show that the sequence $(g_k)_{k<\omega}$ converges to $y$. Indeed,
\begin{equation*}
    \|y - g_k\| \leq \|y_k-y\| + \sum_{n=0}^{k}(n+2)^{-k}\leq \|y_k-y\| + \textstyle{\frac{k+1}{2^{k}}}\to0.
\end{equation*}

Since $(g_k)_{k<\omega}$ is, in particular, bounded, it remains to prove that it is overcomplete in $X$. Suppose, on the contrary, that there exists a subsequence $(g_{k_i})_{i<\omega}$ and a functional $e^*\in X^*$ such that $\langle e^*,g_{k_i}\rangle=0$ for each $i<\omega$. Since $(g_{k_i})_{i<\omega}$ converges to $y$, we have $\langle e^*,y\rangle=0$. Moreover, for every $i<\omega$, (\ref{eq: g_k overcomplete for normed}) allows us to write
\begin{equation*}
    x_0=2^{k_i}\left(g_{k_i}-y_{k_i}- \sum_{n=1}^{k_i} (n+2)^{-k_i}x_n \right).
\end{equation*}
Therefore, we get
\begin{equation*}
\begin{split}
    |\langle e^*,x_0\rangle|&\leq 2^{k_i} |\langle e^*,y_{k_i}-y \rangle|+ \sum_{n=1}^{k_i}\textstyle{(\frac{2}{n+2})^{k_i}} |\langle e^*,x_n\rangle|\\
    &\leq 2^{k_i}\|e^*\|\|y_{k_i}-y\| + \textstyle{\left(\frac{2}{3}\right)^{k_i}}k_i\|e^*\|\\
    & \leq \|e^*\|\textstyle{\frac{2^{k_i}}{k_i!}}+ \textstyle{\left(\frac{2}{3}\right)^{k_i}}k_i\|e^*\|.
    \end{split}
\end{equation*}
Since this holds for any $i<\omega$, we deduce $\langle e^*,x_0\rangle=0$. Suppose now, by induction, that for each $p<j$ we have $\langle e^*,x_p\rangle=0$. As above, from (\ref{eq: g_k overcomplete for normed}) we obtain, for every $i<\omega$ with $k_i>j$,
\begin{equation*}
    x_j=(j+2)^{k_i}\left(g_{k_i}-y_{k_i}- \sum_{n=0}^{j-1} (n+2)^{-k_i}x_n - \sum_{n=j+1}^{k_i} (n+2)^{-k_i}x_n \right).
\end{equation*}
The inductive assumption and $\langle e^*,y\rangle=0$ then yield us 
\begin{equation*}
   \begin{split}
    |\langle e^*,x_j\rangle|&\leq (j+2)^{k_i} |\langle e^*, y_{k_i}-y\rangle| + \sum_{n=j+1}^{k_i} \textstyle{\left(\frac{j+2}{n+2}\right)^{k_i}}|\langle e^*,x_n\rangle|\\
    &\leq \|e^*\|\textstyle{\frac{(j+2)^{k_i}}{k_i!}} + \textstyle{\left(\frac{j+2}{j+3}\right)^{k_i}}k_i\|e^*\|.
    \end{split}
\end{equation*}
Since the last inequality holds for every $i<\omega$, we get $\langle e^*,x_j\rangle=0$. It follows that $\langle e^*,x_n\rangle =0$ for each $n<\omega$, which is equivalent, by the completeness of $(x_n)_{n<\omega}$, to $e^*=0$. Therefore, the sequence $(g_n)_{n<\omega}$ is an overcomplete sequence in $X$, that is not relatively norm compact in $X$, as desired.
\end{proof}

\begin{remark} The above proof is based on a modification of Bra\ss' argument for constructing overcomplete sequences. It is easy to see that the argument can be modified by using geometric sequences, more in the spirit of Klee's proof. Indeed, let $(x_n)_{n<\omega}$ be a normalised complete sequence for a normed space $X$ and let $(\lambda_n)_{n<\omega}\subseteq(0,1)$ be a decreasing sequence with $\lambda_n\searrow0$. Pick $y\in\widehat{X}\setminus X$ and a sequence $(y_n)_{n<\omega}\subseteq X$ such that $\|y_n-y\|\to0$ sufficiently fast. More precisely, we require that, for every $j<\omega$,
\begin{equation*}
    \frac{1}{(\lambda_n)^j} \|y_n-y\|\to0 \;\quad \text{as}\;\; n\to\infty.
\end{equation*}
Then, the same argument as in the previous proof easily shows that the vectors
\begin{equation*}
    g_k:=y_k+ \sum_{j=0}^k (\lambda_k)^{j+1} x_j \qquad (k<\omega)
\end{equation*}
form an overcomplete sequence in $X$, with $g_k\to y$.
\end{remark}

In conclusion of our article, let us present a proof of the aforementioned result \cite[Theorem 2.1]{FonfZanco} that every bounded overcomplete (or, more generally, almost overcomplete) sequence in a Banach space is relatively compact. Let us recall that a sequence is \emph{almost overcomplete} in $X$ if the closed linear span of every its subsequence has finite codimension in $X$.

The argument below is only a minor modification of the one in \cite{FonfZanco}, but we present it because it seems conceptually simpler (in particular, it does not use the LUR renormability of separable Banach spaces and the two cases of their article) and it allows us to distill the following lemma, that was only implicit in \cite{FonfZanco}.

\begin{lemma} Let $(x_n)_{n<\omega}$ be an almost overcomplete sequence in a normed space $X$. If $x_n\to x$ weakly, then it converges in norm.
\end{lemma}

\begin{proof} Assume, by contradiction, that there exist $\e>0$ and a subsequence $(x_{n_k})_{k<\omega}$ such that $\|x_{n_k}-x\|\geq\e$ for $k<\omega$. Therefore, up to a further subsequence, we can assume that $(x_{n_k} -x)_{k<\omega}$ is a basic sequence, whence ${\rm codim}\,\overline{\rm span}\{x_{n_{2k}}-x\}_{k<\omega}= \infty$. It readily follows that ${\rm codim}\,\overline{\rm span}\{x_{n_{2k}}\}_{k<\omega}= \infty$, a contradiction.
\end{proof}

\begin{corollary}[\cite{FonfZanco}] Every bounded almost overcomplete sequence $(x_n)_{n<\omega}$ in a Banach space $X$ is relatively compact.
\end{corollary}
\begin{proof} Since every subsequence of $(x_n)_{n<\omega}$ is almost overcomplete, it suffices to prove that $(x_n)_{n<\omega}$ admits a convergent subsequence. Notice first that $(x_n)_{n<\omega}$ is relatively weakly compact, because, if not, it would admit a basic subsequence (see, \emph{e.g.}, \cite[Theorem 1.5.6]{ak}), a contradiction. Therefore, by the Eberlein--\v{S}mulian theorem, $(x_n)_{n<\omega}$ admits a weakly convergent subsequence, which, by the above lemma, converges in norm.
\end{proof}

{\bf Acknowledgements.} We are most grateful to Bence Horv\'{a}th for suggesting us the proof of Claim \ref{Claim WCG}; our previous argument was unnecessarily complicated.


\end{document}